\documentclass[twoside, reqno]{amsart}

\usepackage[left=2.5cm, right=2.5cm, top=3cm, bottom=2.5cm]{geometry}

\usepackage[colorlinks=true, pdfstartview=FitV, linkcolor=blue,citecolor=blue, urlcolor=blue]{hyperref}

\usepackage[usenames]{xcolor}
\definecolor{labelkey}{rgb}{0,0,1}
\definecolor{Red}{rgb}{0.7,0,0.1}
\definecolor{Green}{rgb}{0,0.7,0}

\usepackage{amsfonts, amssymb, amsmath, amsthm, mathrsfs, bbm, cjhebrew, gensymb, textcomp, mathtools, dsfont,tikz}

\usepackage[normalem]{ulem}

\usepackage{commath, setspace, subcaption, parcolumns, multirow, multicol, accents, comment, marginnote, verbatim, empheq, enumerate, stackrel, enumitem}

\usepackage[capitalize,nameinlink,noabbrev]{cleveref}
\usepackage{graphicx}
\usepackage{epsfig}
\usepackage{psfrag}
\usepackage{float}
\usepackage[active]{srcltx}
\usepackage[pagewise, mathlines]{lineno}

\newtheorem{theorem}{Theorem}[section]

\newtheorem{Thm}{Theorem}[section]
\newtheorem{Prop}{Proposition}[section]
\newtheorem{Lem}[theorem]{Lemma}
\newtheorem{Rmk}{Remark}[section]

\numberwithin{equation}{section}

\newcommand{\be}{\beta}
 
\newcommand{\De}{\Delta}
\newcommand{\eps}{\epsilon}

\newcommand{\kap}{\kappa}

\newcommand{\Lam}{\Lambda}
\newcommand{\si}{\sigma}
\newcommand{\xe}{\xi-\eta}

\newcommand{\tht}{\theta}

\newcommand{\om}{\omega}

\newcommand{\ze}{\zeta}

\newcommand{\lpj}{\triangle_j}

\newcommand{\ZZ}{\mathbb{Z}}
\newcommand{\RR}{\mathbb{R}}

\newcommand{\lb}{\big\langle}
\newcommand{\rb}{\big\rangle}

\newcommand{\goesto}{\rightarrow}

\newcommand{\Sob}[2]{\lVert#1\rVert_{#2}}
\newcommand{\nrm}[1]{\lVert#1\rVert}

\newcommand{\bdy}{\partial}

\newcommand{\Ae}{\abs{\eta}}
\newcommand{\Ax}{\abs{\xi}}

\newcommand{\Hdot}{\dot{H}}

\newcommand{\Acal}{\mathcal{A}}
\newcommand{\Bcal}{\mathcal{B}}

\DeclareMathOperator{\supp}{supp}

\DeclareMathOperator*{\Div}{div}
\title[On local well-posedness of log regularizations of generalized SQG in borderline spaces]{On local well-posedness of logarithmic regularizations of generalized SQG equations in borderline Sobolev spaces}
\author{Michael S. Jolly, Anuj Kumar, Vincent R. Martinez}
\thanks{The research of M.S.J. and A.K. was supported in part by the NSF grant DMS-1818754. The research of V.R.M. was supported in part by the PSC-CUNY grant 64335-00 52.}

\begin{document}

\maketitle
\begin{abstract}
	This paper studies a family of generalized surface quasi-geostrophic (SQG) equations for an active scalar $\theta$  on the whole plane  whose velocities have been mildly regularized, for instance, logarithmically. The well-posedness of these regularized models in borderline Sobolev regularity have previously been studied by D. Chae and J. Wu when the velocity $u$ is of lower singularity, i.e., $u=-\nabla^{\perp}\Lam^{\be-2}p(\Lam)\tht$, where $p$ is a logarithmic smoothing operator and $\be \in [0,1]$. We complete this study by considering the more singular regime $\be\in(1,2)$. The main tool is the identification of a suitable linearized system that preserves the underlying commutator structure for the original equation. We observe that this structure is ultimately crucial for obtaining continuity of the flow map. In particular, straightforward applications of previous methods for active transport equations fail to capture the more nuanced commutator structure of the equation in this more singular regime.  The proposed linearized system nontrivially modifies the flux of the original system in such a way that it coincides with the original flux when evaluated along solutions of the original system. The requisite estimates are developed for this modified linear system to ensure its  well-posedness. 
\end{abstract}

{\noindent \small {\it {\bf Keywords: surface quasi-geostrophic (SQG)
      equation, generalized SQG equation, borderline spaces, Hadamard well-posedness, existence, uniqueness, continuity with respect to initial data, logarithmic regularization}
  }  }

\section{Introduction}
    In this paper, we study the initial value problem for a mildly regularized class of inviscid generalized surface quasi-geostrophic equations over the whole plane $\mathbb{R}^2$:
    \begin{align}\label{E:log-beta}
	\begin{cases}
	\partial_{t}\theta +u \cdot \nabla \theta=0,\\
	u=\nabla^{\perp}{\psi}:=(-\partial_{x_2}{\psi},\partial_{x_1}{\psi}) ,\quad \Delta {\psi}=\Lambda^{\beta}p(\Lam)\theta,\quad 0\le\beta<2,\\
	\tht(0,x)=\tht_{0}(x).
	\end{cases}
	\end{align}
	Here, $\theta$ represents the evolving scalar and $\psi$ its corresponding streamfunction. The operator $\Lam$ is the fractional laplacian operator, $(-\Delta)^{\frac{1}{2}}$ and $p(\Lam)$ is any Fourier multiplier operator that satisfies
	\[
	    \mathcal{F}(p(\Lam)f)(\xi)=p(|\xi|)\hat{f}(\xi).\]
    We will assume that $p(\cdot)$ is a radial function that satisfies the following conditions. 
	\begin{align}\label{def:p}
	    p(r)>0, \quad p\in L^{\infty}(\mathbb{R})\cap C^{\infty}(\mathbb{R}\backslash\{0\}),\quad \sup_{r>0}\frac{r|p'(r)|}{\sqrt{p(r)}}<\infty,\quad
	    \int_{1}^{\infty}\frac{p^{2}(r)}{r}dr<\infty.
	\end{align}
	In particular, note that 
	\begin{align}\label{E:log:multiplier}
	    p(|\xi|)=\ln(e+|\xi|^2)^{-\mu},\quad \mu>\frac{1}{2},
	\end{align}
	satisfies the above conditions. Consequently, \eqref{E:log-beta} contains the logarithmically regularized counterparts of the generalized SQG equation. In specific, the corresponding generalized SQG equation is given by
	\begin{align}\label{E:beta}
	\begin{cases}
	\partial_{t}\theta +u \cdot \nabla \theta=0,\\
	u=\nabla^{\perp}{\psi},\quad \Delta {\psi}=\Lambda^{\beta}\theta,\quad 0\le\beta<2.
	\end{cases}
	\end{align}

	For $\beta \in [0,1]$, \eqref{E:beta} interpolates between the 2D incompressible Euler vorticity equation ($\be=0$) and the SQG equation ($\beta=1$), while for $\be \in (1,2)$, it represents a family of active scalar equations with increasingly singular velocity. The Cauchy problem for \eqref{E:beta} was first considered in \cite{ChaeConstantinWu2012b}, where a regularity criterion in terms of the norm of $\tht$ in the H\"older space $C^{\be}$ was established. Since then, well-posedness in various functional spaces has been addressed \cite{ChaeConstantinCordobaGancedoWu2012,HuKukavicaZiane2015,LazarXue2019, JollyKumarMartinez2020a}, as well as continuity with respect to $\be$ \cite{YuZhenJiu2019}, global existence of weak solutions on bounded domains \cite{Nguyen2018}, the existence of invariant Gaussian measures for the flow \cite{NahmodPavlovicStaffilaniTotz2018}, and various studies on the corresponding point-vortex models \cite{FlandoliSaal2019, GeldhauserRomito2019, GeldhauserRomito2020,LuoSaal2020}.  One important aspect of the family \eqref{E:beta} is that it allows one to rigorously identify the Euler equation, $\be=0$, as a critical model. Indeed, from the point of view of the so-called `patch problem' posed in the half-space, global regularity holds at $\be=0$, whereas finite-time singularity can occur when $\be>0$ is sufficiently small \cite{KiselevRyzhikYaoZlatos2016}; the patch problem has also recently been studied in \cite{CordobaGomez-SerranoIonescu2019, GancedoPatel2021, KiselevYaoZlatos2017}. 
	
	The issue of local well-posedness of the regularized models, \eqref{E:log-beta} was initially studied by D. Chae and J. Wu in \cite{ChaeWu2012}. There, existence and uniqueness of solutions was established in the borderline Sobolev space, $H^{\be+1}$ when $\be\in[0,1]$. The term `borderline' refers to the threshold of regularity corresponding to solutions of \eqref{E:beta} with respect to which the gradient of the velocity in $L^\infty$ can be controlled; specifically, $\Sob{\tht}{H^{\si}}$ bounds $\Sob{\nabla u}{L^\infty}$, but only when $\si>\be+1$. Thus, the result in \cite{ChaeWu2012} essentially showed that this obstruction to a local theory at the critical level $\si=\be+1$ can be overcome by introducing a regularization of the form in \eqref{E:log-beta}, provided that $\be\in[0,1]$.
	Interestingly, they identify an additional threshold in the degree of this regularization for the local-well posedness to hold, namely $\mu>1/2$ in \eqref{E:log:multiplier}. This threshold was later shown to be sharp in the endpoint case, $\be=0$, represented by the logarithmically regularized 2D incompressible Euler equation by H. Kwon \cite{Kwon2020}; specifically, when $\be=0$, the corresponding system was shown to be strongly ill-posed in the borderline space $H^{1}(\mathbb{R}^2)\cap \dot{H}^{-1}(\mathbb{R}^2)$. The result of Kwon is an extension of the seminal paper of J. Bourgain and D. Li \cite{BourgainLi2015}, in which a longstanding open problem of whether the initial value problem to the $d$-dimensional Euler equation was well-posed or not in the scaling-critical topology (for the velocity) $H^{1+d/2}(\mathbb{R}^d)$, for $d=2,3$, was resolved by demonstrating the existence of a norm inflation phenomenon via ``large Lagrangian deformation" that yields non-continuity of the corresponding data-to-solution map. A notable alternative approach to demonstrating ill-posedness was established by T. Elgindi and I. Jeong \cite{ElgindiJeong2017} in the two-dimensional setting and by T. Elgindi and N. Masmoudi \cite{ElgindiMasmoudi2020} in higher dimensions. To the best of our knowledge, the issue of well-posedness of \eqref{E:beta} in the scaling-critical topology, $H^{\be+1}(\mathbb{R}^2)$ remains open. On the other hand, ill-posedness for active scalar equations with an even more singular constitutive law \cite{KukavicaVicolWang2016} as well as non-uniform continuity of the data-to-solution operator \cite{HimonasMisolek2010,MisiolekYoneda2012,MisiolekYoneda2018} have been studied. In this paper, the ``positive" side of this issue is treated as we establish the analogous results of D. Chae and J. Wu in the borderline Sobolev spaces for the more singular range $\be\in(1,2)$, thereby providing a complete picture of well-posedness for the full family of mildly regularized inviscid gSQG equations. In particular, we identify a similar threshold for well-posedness through the condition $\mu>1/2$ in \eqref{E:log:multiplier}.
	
	As previously mentioned, we  establish local well-posedness (in the sense of Hadamard) for the initial value problem \eqref{E:log-beta} in the range $\be \in (1,2)$ in the borderline Sobolev space $H^{\be+1}(\mathbb{R}^2)$. This result is stated in the following theorem.
	\begin{theorem}\label{T:theorem1}
	Let $\beta \in (1,2)$. For each $\tht_{0}\in H^{\be+1}(\mathbb{R}^2)$, there exists $T=T(\Sob{\tht_{0}}{H^{\be+1}})>0$ and a unique solution, $\tht$, of \eqref{E:log-beta} such that
	\[\tht \in C([0,T];H^{\be+1}(\mathbb{R}^2)).\]
	In particular, the data-to-solution map, $\Phi$, such that
	\begin{align}\label{def:flowmap}
	    \Phi:H^{\be+1}(\mathbb{R}^2)\goesto  \bigcup_{T>0}C([0,T];H^{\be+1}(\mathbb{R}^2)),\quad \tht_{0}\mapsto\tht(t;\tht_0),
	\end{align}
    is well-defined and continuous.
	\end{theorem}\par
    Observe that for $\tht \in H^{\be+1}(\mathbb{R}^2)$, one has
	    \begin{align}\label{eq:borderline}
        	\Sob{\nabla u}{L^{\infty}}\le C\Sob{\nabla \nabla^{\perp}\Lam^{\be-2}p(\Lam)\tht}{L^\infty}\le C\Sob{\tht}{H^{\be+1}}.
        \end{align}
    As a result, when $\be \in [0,1]$, one can use the standard estimates for the transport equation in the Sobolev space $H^{\be+1}(\mathbb{R}^2)$ to obtain existence, uniqueness and continuity of the flow map for \eqref{E:log-beta}. This approach, however, no longer seems to directly work when $\be \in (1,2)$ as the estimates for the transport equation in $H^{\be+1}$ instead require control of $\Sob{\nabla u}{H^{\be}}$ (see, for instance, \cite[Theorem 3.19]{BahouriCheminDanchinBook2011}).
	To overcome this difficulty, we observe that one must exploit the more nuanced commutator structure within  \eqref{E:log-beta}. Such a structure was originally identified in \cite{ChaeConstantinCordobaGancedoWu2012}, where it was exploited to demonstrate local well-posedness in $H^4(\mathbb{R}^2)$; this result was subsequently improved in \cite{HuKukavicaZiane2015} to the space $H^{\be+1+\eps}(\mathbb{R}^2)$, for $\eps>0$. We ultimately observe that the main obstruction to well-posedness for the models \eqref{E:log-beta} lies not in establishing existence and uniqueness, but rather in continuity of the data-to-solution map. Although the nuanced commutator structure in \eqref{E:log-beta} is crucial for establishing existence, exploiting it to demonstrate continuity is more delicate.
	This is ultimately done by identifying a suitable perturbation of the flux (see \eqref{def:mod:flux}) in \eqref{E:log-beta}, viewed in divergence-form, then developing the proper apriori estimates for linearizations of this perturbed system naturally pertaining to the argument for continuity (see \eqref{E:mod:claw}). These linearizations are ultimately obtained by adopting a classical splitting scheme that was introduced by Kato in \cite{Kato1975} to establish continuity of flow maps for quasilinear symmetric systems, but adapted for the system \eqref{E:log-beta} (see \eqref{E:omega} and \eqref{E:zeta}); we refer the reader to \cite{BahouriCheminDanchinBook2011} for details. The main ingredient for the proof of continuity is the stability estimate for equation \eqref{E:mod:claw} as stated in \cref{lem:stability}. The linear equation can then be viewed as a conservation law \eqref{E:mod:claw} whose flux incorporates the commutator structure of the original system. We remark that similar ideas were recently exploited by the authors in \cite{JollyKumarMartinez2020a} to address existence, uniqueness, and instantaneous smoothing of solutions to the ``supercritically" dissipative counterpart of \eqref{E:beta}, i.e., with $\Lam^\kap\tht$ for $0<\kap<1$, for large initial data belonging to the corresponding scaling-critical Sobolev spaces, $H^{\be+1-\kap}(\mathbb{R}^2)$.

	\begin{Rmk}
	By using a similar approach for the equation in \eqref{E:beta}, we can establish the local well-posedness (in the sense of Hadamard) of \eqref{E:beta} for $\be \in (1,2)$ in the smaller critical space $B^{\be+1}_{2,1}(\mathbb{R}^2)$. It would also be interesting to establish these results in the more general setting of $L^p$-based Besov spaces in the spirit of \cite{Vishik1998}. We refer the reader to the thesis of the second author for additional details (\cite{KumarThesis2021}).
	\end{Rmk}

	\section{Mathematical Preliminaries}\label{section:notation:preliminaries}
	For the rest of the paper, we assume that $C$ denotes a positive constant whose value may change from step to step. Dependence on other parameters may be specified when relevant. Moreover, all single integrals will occur over $\mathbb{R}^2$, unless otherwise specified.
	
	Let $\mathscr{S}(\RR^2)$ denote the space of Schwartz class functions defined on $\RR^2$ and $\mathscr{S}'(\RR^2)$ denote the space of tempered distributions. For $f\in\mathscr{S}'(\RR^2)$, we denote by $\hat{f}$ or $\mathcal{F}(f)$, the Fourier transform of $f$, defined as
	\[\hat{f}(\xi):=\int e^{-2\pi i x\cdot \xi}f(x)dx.\] Recall that $\mathcal{F}$ is an isometry on $L^2$ and satisfies
	    \begin{align}\notag
	        \lb f,g\rb=\lb \hat{f},\hat{g}\rb.
	    \end{align}
	We denote by $\Lam^\si$, $\si\in\RR$, the fractional laplacian operator, defined as
	    \begin{align}\notag
	        \mathcal{F}(\Lam^\si f)(\xi)=|\xi|^\si\mathcal{F}(f).
	    \end{align} 
    We recall the definition of the Fourier-based homogeneous and inhomogeneous Sobolev spaces on $\mathbb{R}^{2}$. For $\si \in \mathbb{R}$, we have
	    \begin{align}
	        &\Hdot^\si(\mathbb{R}^2):=\left\{f\in \mathscr{S}(\RR^2):\hat{f}\in L^2_{loc}, \Sob{f}{\Hdot^\si}:=\Sob{\Lam^\si f}{L^2}<\infty\right\},\label{def:hom:Sob:norm}\\
	        &H^\si(\mathbb{R}^2):=\left\{f\in \mathscr{S}(\RR^2):\hat{f}\in L^2_{loc}, \Sob{f}{H^\si}:=\Sob{(I-\De)^{\si/2}) f}{L^2}<\infty\right\}.\label{def:inhom:Sob:norm}
	    \end{align}

Hereafter, we will suppress the expression of the domain $\mathbb{R}^2$ when denoting the Schwartz, Sobolev, or related spaces, except when we would like to emphasize the dimensionality in the statement.
	
We now provide a brief review of the Littlewood-Paley decomposition and refer the reader to \cite{BahouriCheminDanchinBook2011, Chemin1998} for additional details. We define
	\begin{align*}
	\mathscr{Q}(\RR^2):=\left\{f\in \mathscr{S}(\RR^2): \int f(x)x^{\tau}\, dx=0, \quad \abs{\tau}=0,1,2,\cdots \right\}.
	\end{align*}
	Let $\mathscr{Q}(\RR^2)'$ denote the topological dual of $\mathscr{Q}(\RR^2)$. Then, $\mathscr{Q}(\RR^2)'$ can be identified with the space of tempered distributions modulo the vector space of polynomials on $\mathbb{R}^2$, denoted by $\mathscr{P}$, i.e.
	\begin{align*}
	\mathscr{Q}'(\RR^2)\cong\mathscr{S}(\RR^2)/\mathscr{P}.
	\end{align*}
	We will denote by ${\Bcal}(r)$, the open ball of radius $r$ with center at the origin and by ${\Acal}(r_{1},r_{2})$, the open annulus with inner and outer radii $r_{1}$ and $r_{2}$, and with center at the origin. It can be shown that there exist two non-negative radial functions $\chi,\phi\in\mathscr{S}(\RR^2)$ with $\supp\chi\subset{\Bcal}(1)$ and $\supp\phi\subset{\Acal}(2^{-1},2)$ such that for $\phi_j(\xi):=\phi(2^{-j}\xi)$ and $\chi_j(\xi):=\chi(2^{-j}\xi)$, following conditions are satisfied
	\begin{align*}
	    \begin{cases}
	    \sum_{j\in\ZZ}\phi_j(\xi)=1,\\
	    \chi+\sum_{j\geq0}\phi_j\equiv 1,\,\forall \xi\in\RR^2\setminus\{\mathbf{0}\},\\
	    \supp\phi_i\cap\supp\phi_j=\varnothing,\,\text{if}\,
	    |i-j|\geq2,\\
	    \text{and}\quad\supp\phi_i\cap\supp\chi =\varnothing.
	    \end{cases}
	\end{align*}
	We will fix the following notation
        \begin{align}\notag
            {\Acal}_{j}= {\Acal}(2^{j-1},2^{j+1}),\quad{\Acal}_{\ell,k}= {\Acal}(2^{\ell},2^{k}),\quad  {\Bcal}_j={\Bcal}(2^j).
        \end{align}
    With this notation, note that
	    \begin{align}\label{eq:rewrite:supp}
	        \supp\phi_j\subset{\Acal}_j,\quad \supp\chi_j\subset{\Bcal}_j.
	    \end{align}
	Denote by ${\lpj}$ and $S_{j}$, the (homogeneous) Littlewood-Paley dyadic blocks which are defined via Fourier transform as
	\begin{align}\notag
	\mathcal{F}({\lpj}f)=\phi_{j}\mathcal{F}(f), \quad \mathcal{F}(S_{j}f)=\chi_{j}\mathcal{F}(f). 
	\end{align}
    Observe that owing to \eqref{eq:rewrite:supp}, we obtain
	\begin{align*}
	&\mathcal{F}({\lpj}f)|_{{\Acal}_j^c}=0,\quad
	\mathcal{F}(S_{j}f)|_{{\Bcal}_j^{c}}=0,
	\end{align*}
    Also observe that for any $f \in \mathscr{S}(\RR^2)$, we have
    \begin{align*}
         f&=S_if+\sum_{j\geq i}\lpj f,\quad i\in\ZZ.
    \end{align*}
	On the other hand, when $f\in\mathscr{Q}(\RR^2)'$, we have
	    \begin{align*}
	        f&=\sum_{j\in\ZZ}\lpj f.
	    \end{align*}
Recall that the Besov seminorm $\norm{\cdotp}_{\dot{B}^{\si}_{2,2}}$, is defined in terms of the Littlewood-Paley dyadic blocks as
	    \begin{align*}
	        \nrm{f}_{\dot{B}^{\si}_{2,2}}:=\left(\sum_{j\in \mathbb{Z}}\left(2^{j\si}\nrm{\lpj f}_{L^2}\right)^{2}\right)^{\frac{1}{2}}.
	    \end{align*}
	In particular, we have the following characterization of the Sobolev seminorms
	    \begin{align*}
	       C^{-1}\Sob{f}{\dot{B}^\sigma_{2,2}(\mathbb{R}^2)}\leq \Sob{f}{\dot{H}^\si(\mathbb{R}^2)}\leq C\Sob{f}{\dot{B}^\sigma_{2,2}(\mathbb{R}^2)},
	    \end{align*}
	for some constant $C$ depending only on $\si$. We will frequently employ the following well known inequality (see  \cite{BahouriCheminDanchinBook2011},\cite{Chemin1998}) which quantifies the relation between the dyadic blocks and the fractional laplacian operator.
	\begin{Lem}[Bernstein inequalities]\label{T:Bernstein}
		Let $\si\in\RR$ and $1\le p \le q\le \infty$. Then
		\begin{align*}
	C^{-1}2^{\si j}\nrm{{\lpj}f}_{L^q(\mathbb{R}^2)}\le \nrm{\Lam^{\si}{\lpj}f}_{L^q(\mathbb{R}^2)}\le C 2^{\si j+2j(\frac{1}{p}-\frac{1}{q})}\nrm{{\lpj}f}_{L^p(\mathbb{R}^2)},
		\end{align*}
	 where $C>0$ is a constant that depends on $p,q$ and $\si$.
	\end{Lem}
Let us recall the following classical product estimate in homogenous Sobolev spaces (\cite{RunstSickel1996}).
\begin{Lem}\label{T:Sobolev}
Suppose that $s,t<1$ and $s+t>0$. Let $f\in \Hdot^{s}(\mathbb{R}^2)$ and $g\in \Hdot^{t}(\mathbb{R}^2)$. Then
\begin{align*}
    \Sob{fg}{\Hdot^{s+t-1}}\le C\Sob{f}{\Hdot^s}\Sob{g}{\Hdot^t}.
\end{align*}
\end{Lem}
We will also use a particular dualized version of the above inequality as stated in \cite{JollyKumarMartinez2020a}.
\begin{Lem}\label{lem:commutator1}
		For $\si \in (-1,1)$ and $f,g,h \in \mathscr{S}(\mathbb{R}^2)$, define
		\begin{align*}
		\mathcal{L}_{\si}(f,g,h)&:=\iint\Ax^{\si}\hat{f}(\xi-\eta){ \hat{g}(\eta)\overline{\hat{h}(\xi)}} d\eta d\xi.
		\end{align*}
		Suppose that $\supp \hat{h}\subset\Acal_j$, for some $j\in\mathbb{Z}$. Then for each $\si\in(-1,1)$ and $\epsilon \in (0,2)$ such that $\sigma>\epsilon-1$, there exists a constant $C>0$, depending only on $\si,\epsilon$, and $\{c_j\}\in\ell^2(\mathbb{Z})$ with $\Sob{\{c_j\}}{\ell^2}\leq1$ such that
		\begin{align}\notag
		    	|\mathcal{L}_{\si}(f,g,h)|\le  Cc_j2^{\epsilon j}\min\left\{\nrm{f}_{\Hdot^{1-\epsilon}}\nrm{g}_{\Hdot^{\si}},\nrm{g}_{\Hdot^{1-\epsilon}}\nrm{f}_{\Hdot^{\si}}\right\}\nrm{h}_{L^{2}}.
		\end{align}
	\end{Lem}

\section{Commutator estimates}
In this section, we will establish the commutator estimates that will be required in order to prove \cref{T:theorem1}. First, we prove  \cref{lem:commutator2}, which establishes an estimate for the localized commutator of the operator $\nabla^{\perp}\Lam^{\beta-2}p(\Lam)$ that defines the velocity $u$ in terms of the scalar $\theta$ in \eqref{E:log-beta}, where we recall $p$ is a multiplier satisfying \eqref{def:p}. To this end, we consider a multiplier $P(D)$ such that
    \begin{align}\label{E:p:conditions}
     \sup_{\xi}|P(\xi)|,\quad \sup_{\xi}|\xi||\nabla P(\xi)|<\infty.
    \end{align}
    
    We will denote the commutator of two operators, $S$ and $T$, by $[S,T]$, where
    \[[S,T]:=ST-TS.\]
\begin{Lem}\label{lem:commutator2}
Let $s\in(0,1)$, $\eps \in [0,1)$ be such that $\eps+s\le 1$. Let $f\in \Hdot^{\eps}(\mathbb{R}^2)$, $g\in \Hdot^{2-s-\eps}(\mathbb{R}^2)$ and $h\in L^{2}(\mathbb{R}^2)$. Let $P$ be any Fourier multiplier satisfying \eqref{E:p:conditions}. Given $k>0$, suppose that $\supp \hat{f}\subset\Acal_i$ and $\supp \hat{h}\subset\Acal_j$, where $\abs{i-j}\le k$. There exists a constant $C>0$, depending only on $s,k,\eps$, such that
    \begin{align}
       |\langle [\Lam^{-s}P(D)\bdy_\ell,g]f,h\rangle|\leq C\Sob{g}{\dot{H}^{2-s-\eps}}\Sob{ f}{\Hdot^\eps}\Sob{h}{L^2},\quad \ell=1,2.\notag
    \end{align}
\end{Lem}

The proof of \cref{lem:commutator2} is similar to that of Lemma 4.3 in \cite{JollyKumarMartinez2020a}. To prove it, we will make use of the following convexity-type inequality that is proved in \cref{sect:app}.
\begin{Lem}\label{lem:elem:convex}
Let $\varphi,\vartheta\in\mathbb{R}^d$, where $d\geq1$, such that $|\vartheta|=1$. For all $0<s<1$, there exists a constant $C>0$ depending only on $s$ such that
    \begin{align}\notag
    \int_0^1\frac{1}{|\varphi+\tau\vartheta|^s}d\tau\leq C.
    \end{align}
\end{Lem}

\begin{proof}[Proof of \cref{lem:commutator2}]
We define the functional
    \begin{align}\label{def:L:functional:1}
        \mathcal{L}_{s,\ell}(f,g,h):=\iint m_{s,\ell}(\xi,\eta)\hat{f}(\xi-\eta)\hat{g}(\eta)\overline{\hat{h}(\xi)}d\eta d\xi,
    \end{align}
where
    \begin{align*}
        m_{s,\ell}(\xi,\eta):=|\xi|^{-s}P(\xi)\xi_\ell-|\xi-\eta|^{-s}P(\xi-\eta)(\xi-\eta)_\ell,
    \end{align*}
and observe that
    \begin{align}\notag
      \lb [\Lam^{-s}P(D)\bdy_\ell, g]f,h\rb=\mathcal{L}_{s,\ell}(f,g,h).
    \end{align}
Now let
    \begin{align}\label{def:A}
    \mathbf{A}(\tau,\xi,\eta):=\tau\xi+(1-\tau)(\xe)=(\xi-\eta)+\tau\eta=\xi-(1-\tau)\eta.
    \end{align}
For convenience, we will suppress the dependence of $\mathbf{A}$ on $\xi,\eta$ for the remainder of the proof. Observe that
    \begin{align}\label{E:meanvalue1}
        &|m_{s,\ell}(\xi,\eta)|\notag\\&=\abs{\int_{0}^{1}\frac{d}{{d\tau}}\left(\abs{\mathbf{A}(\tau)}^{-s}P(\mathbf{A}(\tau))\mathbf{A}(\tau)_{\ell}\right)d\tau}\notag\\
        &=\abs{\int_{0}^{1}\left(-s\abs{\mathbf{A}(\tau)}^{-s-2}(\mathbf{A}(\tau)\cdot \eta)P(\mathbf{A}(\tau))\mathbf{A}(\tau)_{\ell}+\abs{\mathbf{A}(\tau)}^{-s}(\nabla P(\mathbf{A}(\tau))\cdot \eta)\mathbf{A}(\tau)_{\ell}+\abs{\mathbf{A}(\tau)}^{-s}P(\mathbf{A}(\tau))\eta_{\ell}\right)d\tau}\notag\\
    &\le C\abs{\eta}\int_{0}^{1}\abs{\mathbf{A}(\tau)}^{-s}d\tau,
    \end{align}
where the fact $s\in(0,1)$ is invoked to obtain the last inequality. 
By assumptions on the supports of $\hat{f}$ and $\hat{h}$, we can assume that $\supp \hat{g}\subset\Bcal_{i+k+2}.$
Using this, we obtain
\[\mathcal{L}_{s,\ell}(f,g,h)=I+II,\]
where
\begin{align*}
    I=&\iint m_{s,\ell}(\xi,\eta)\hat{f}(\xi-\eta)\mathbbm{1}_{{\Bcal}_{i-3}}(\eta)\hat{g}(\eta)\overline{\hat{h}(\xi)}d\eta d\xi,\\
    II=&\iint m_{s,\ell}(\xi,\eta)\hat{f}(\xi-\eta)\mathbbm{1}_{{\Acal}_{i-3,i+k+2}}(\eta)\hat{g}(\eta)\overline{\hat{h}(\xi)}d\eta d\xi.
\end{align*}
Now we treat $I$ and $II$.
\subsubsection*{{Estimating} $I$\nopunct}: For $ \eta \in {\Bcal}_{i-3}$, we have
\begin{align*}
\abs{\mathbf{A}(\tau)} \ge \abs{\xe}-\tau\abs{\eta} \ge 2^{i-1}-2^{i-3} = 3(2^{i-3})\ge \frac{3}{16}\abs{\xe}.
\end{align*}
Thus
\begin{align*}
 |I| \le C\iint\abs{\xe}^{-s}|\eta|   \mathbbm{1}_{\Bcal_{i-3}}(\eta)|\hat{g}(\eta)||\hat{f}(\xi-\eta)||{\hat{h}(\xi)}|d\eta d\xi.
\end{align*}
By the Cauchy-Schwarz inequality, Young's convolution inequality, and Plancherel's theorem, we obtain
\begin{align}\label{est:I:intermediate}
 |I|\le C\Sob{|\eta|\mathbbm{1}_{\Bcal_{i-3}} \hat{g}}{L^{\frac{4}{4-s}}}\Sob{|\eta|^{-s}\hat{f}}{L^{\frac{4}{2+s}}}\Sob{h}{L^{2}}.
\end{align}
By H\"older's inequality, we have
\begin{align*}
\Sob{|\eta| \mathbbm{1}_{\Bcal_{i-3}}(\eta)\hat{g}(\eta)}{L^{\frac{4}{4-s}}}&\le \Sob{\mathbbm{1}_{\Bcal_{i-3}}|\eta|^{s+\eps-1}}{L^{\frac{4}{2-s}}}\Sob{|\eta|^{2-s-\eps}\hat{g}(\eta)}{L^{2}}
= C2^{i(\frac{s}{2}+\eps)} \Sob{g}{\Hdot^{2-s-\eps}},\\
\Sob{|\eta|^{-s}\hat{f}(\eta)}{L^{\frac{4}{2+s}}}&\le \Sob{\mathbbm{1}_{\Acal_i}|\eta|^{-s-\eps}}{L^{\frac{4}{s}}}\Sob{|\eta|^{\eps}\hat{f}(\eta)}{L^{2}}
\le 
C2^{i(-\frac{s}{2}-\eps)}\Sob{f}{\Hdot^\eps}.
\end{align*}
Upon returning to (\ref{est:I:intermediate}), we obtain 
\begin{align*}
 |I|\le C\Sob{g}{\Hdot^{2-s-\eps}}\Sob{f}{\Hdot^\eps}\Sob{h}{L^{2}},
\end{align*}
for some constant $C>0$, depending on $s,\eps$.

\subsubsection*{{Estimating} $II$\nopunct}: Let $\varphi:=\frac{\xe}{\Ae}$ and $\vartheta:=\frac{\eta}{\Ae}$.  We observe that for fixed $\xi$ and $\eta$, we have
\begin{align*}
\int_{0}^{1}\abs{\eta}\abs{ \mathbf{A}(\tau)}^{-s}\,d\tau=\Ae^{1-s}\int_{0}^{1}\frac{1}{\abs{\varphi+\tau \vartheta}^{s}}\,d\tau.
\end{align*}
By \cref{lem:elem:convex}, it follows that
\begin{align}\label{E:A-integral}
\int_{0}^{1}\abs{\eta}\abs{ \mathbf{A}(\tau)}^{-s}\,d\tau\le C\Ae^{1-s}.
\end{align}
Hence
 \begin{align*}
 |II| \le C\iint|\eta|^{1-s}   \mathbbm{1}_{\Acal_{i-3,i+k+2}}(\eta)|\hat{g}(\eta)||\hat{f}(\xi-\eta)||{\hat{h}(\xi)}|d\eta d\xi.
\end{align*}
Applying the Cauchy-Schwarz inequality, Young's convolution inequality, Plancherel's theorem, and the Cauchy-Schwarz inequality a second time yields
\begin{align}
|II|\le & C \Sob{|\eta|^{1-s-\eps}\mathbbm{1}_{\Acal_{i-3,i+k+2}}\hat{g}}{L^1}\Sob{|\eta|^{\eps}\hat{f}}{L^2}\Sob{\hat{h}}{L^{2}}\label{est:II:intermediate}\\
\le& C_{k}\Sob{|\eta|^{2-s-\eps}\mathbbm{1}_{\Acal_{i-3,i+k+2}}\hat{g}}{L^{2}}\Sob{f}{\Hdot^\eps}\Sob{h}{L^{2}}\notag \\
\le & C_{k}\Sob{g}{\Hdot^{2-s-\eps}}\Sob{f}{\Hdot^\eps}\Sob{h}{L^2}\notag.
\end{align}
\end{proof}
By slightly modifying the proof of \cref{lem:commutator2}, we obtain the following non-localized form of the above commutator estimate. We point out that similar estimates were also established in \cite{ChaeConstantinCordobaGancedoWu2012} and \cite{JollyKumarMartinez2020a}.
\begin{Lem}\label{lem:commutator3}
Let $s,\eps\in(0,1)$ be such that $\eps+s\le 1$, and $P$ be any Fourier multiplier satisfying (\ref{E:p:conditions}). Suppose that either $(f,g,h) \in L^{2}(\mathbb{R}^2)\times H^{2-s+\eps}(\mathbb{R}^2)\times L^{2}(\mathbb{R}^2)$ or $(f,g,h)\in H^{\eps}(\mathbb{R}^2)\times H^{2-s}(\mathbb{R}^2)\times H^{\eps}(\mathbb{R}^2) $. There exists a constant, $C>0$, depending only on $s,\epsilon$, such that
    \begin{align}
       |\langle [\Lam^{-s}P(D)\bdy_\ell,g]f,h\rangle|\leq C\min\left\{\Sob{g}{{H}^{2-s+\eps}}\Sob{ f}{L^2}\Sob{h}{L^2},\Sob{g}{{H}^{2-s}}\left(\Sob{f}{\Hdot^\eps}\Sob{h}{L^2}+\Sob{h}{\Hdot^\eps}\Sob{f}{L^2}\right)\right\},\quad \ell=1,2.\notag
    \end{align}
\end{Lem}
\begin{proof}
We proceed just as in the proof of \cref{lem:commutator2} and apply (\ref{E:meanvalue1}) and (\ref{E:A-integral}) to obtain
\begin{align}\label{ineq:Plancherel}
    |\langle [\Lam^{-s}P(D)\bdy_\ell,g]f,h\rangle|\le C\iint|\eta|^{1-s}   |\hat{g}(\eta)||\hat{f}(\xi-\eta)||{\hat{h}(\xi)}|d\eta d\xi.
\end{align}
Proceeding as in (\ref{est:II:intermediate}), we obtain
\begin{align*}
    |\langle [\Lam^{-s}P(D)\bdy_\ell,g]f,h\rangle|&\le C\Sob{|\eta|^{1-s}\hat{g}}{L^1}\Sob{f}{L^2}\Sob{h}{L^2}\\
    &\le C_{\eps}\Sob{g}{H^{2-s+\eps}}\Sob{f}{L^2}\Sob{h}{L^2}.
\end{align*}

On the other hand, upon multiplying and dividing by $|\eta|^{\eps}$ in \eqref{ineq:Plancherel}, applying the triangle inequality, $|\eta|^{\eps}\le|\xi|^{\eps}+|\xe|^{\eps}$, then applying the Cauchy-Schwarz inequality and Young's convolution inequality, we obtain
\begin{align*}
    |\langle [\Lam^{-s}P(D)\bdy_\ell,g]f,h\rangle|&\le C\Sob{|\eta|^{1-s-\eps}\hat{g}}{L^1}\left(\Sob{\Lam^{\eps}f}{L^2}\Sob{h}{L^2}+\Sob{f}{L^2}\Sob{\Lam^{\eps}h}{L^2}\right)\\
    &\le C_{\eps}\Sob{g}{H^{2-s}}\left(\Sob{f}{\Hdot^\eps}\Sob{h}{L^2}+\Sob{h}{\Hdot^\eps}\Sob{f}{L^2}\right).
\end{align*}
\end{proof}
Next we prove a commutator estimate for operators which are of the form of a product of Fourier multiplier operators given by $\Lam^{\si}, \bdy_{\ell}, \triangle_j, P(D)$, where we assume that $P$ satisfies (\ref{E:p:conditions}). 
We will let $\mathscr{D}$ denote  
    \begin{align*}
        \mathscr{D}=\Lam\ \text{or}\ \bdy_\ell,\quad \text{for}\ \ell=1,2.
    \end{align*}
For the next result, let us denote by $W(\mathbb{R}^2)$, the space of functions whose Fourier transform belongs to $L^1(\mathbb{R}^2)$.
    
\begin{Lem}\label{lem:commutator4}
Let $s \in [0,1)$, $\nu \in (0,1)$, $\rho \in \mathbb{R}$. Suppose that $\supp\hat{h}\subset{\Acal}_j$, for some $j\in \mathbb{Z}$, and that either $(f,g)\in (\dot{H}^{1-\nu}(\mathbb{R}^2)\times\dot{H}^{s+1}(\mathbb{R}^2))\cup(\dot{H}^s(\mathbb{R}^2)\times\dot{H}^{2-\nu}(\mathbb{R}^2))$ or $(f,\Lam g)\in ((W(\mathbb{R}^2)\cap\dot{H}^1(\mathbb{R}^2))\times\dot{H}^{s}(\mathbb{R}^2))\cup \dot{H}^s(\mathbb{R}^2)\times(W(\mathbb{R}^2)\cap\dot{H}^1(\mathbb{R}^2))$. Let $P$ be a Fourier multiplier symbol satisfying (\ref{E:p:conditions}). Then there exists a sequence $\{c_j\}\in\ell^2(\ZZ)$ such that $\Sob{\{c_j\}}{\ell^2}\leq1$ and
    \begin{align}
        |\lb [\Lam^{s+\rho}P(D){\mathscr{D}}\lpj,g]f,h  \rb|\notag\leq Cc_{j}
        \begin{cases}
              \min\left\{ \Sob{{f}}{\Hdot^{1-\nu}}\Sob{g}{\Hdot^{s+1}},\Sob{{f}}{\Hdot^s}\Sob{g}{\Hdot^{2-\nu}}\right\}\Sob{h}{\Hdot^{\rho+\nu}}    \\
              \min\{(\Sob{\hat{f}}{L^1}+\Sob{f}{\Hdot^1})\Sob{g}{\Hdot^{s+1}},(\Sob{\widehat{\Lam g}}{L^1}+\Sob{g}{\Hdot^2})\Sob{{f}}{\Hdot^{s}}\}\Sob{h}{\Hdot^{\rho}},
        \end{cases}
    \end{align}
for some constant $C>0$, depending only on $s,\rho, \nu$. 
\end{Lem}
\begin{proof}
We will only treat the case of $\mathscr{D}=\bdy_{\ell}$ to avoid redundancy in the argument. The proof for the case $\mathscr{D}=\Lam$ is similar.
First, let us define
    \begin{align*}
        \mathcal{L}^{s}_{j}(f,g,h):=\iint\limits_{\xi\in{\Acal}_j} m_{s,j}(\xi,\eta)\hat{f}(\xi-\eta)\hat{g}(\eta)\overline{\hat{h}(\xi)}\, d\eta \, d\xi,
    \end{align*}
where
    \begin{align*}
        m_{s,j}(\xi,\eta):= \phi_{j}(\xi)P(\xi)\abs{\xi}^{s}\xi_{\ell}-\phi_{j}(\xe)P(\xe)\abs{\xi-\eta}^{s}(\xi-\eta)_{\ell}.
    \end{align*}
Then, using Plancherel's theorem, we see that
    \begin{align}\label{def:L:comm:relation}
        \mathcal{L}^{s}_{j}(f,g,h)=\lb[\Lam^{s}P(D)\mathscr{D}\lpj,g]f,h\rb.
    \end{align}
It is therefore equivalent to obtain bounds on $\mathcal{L}_{j}^{s+\rho}$. 

Let $\mathbf{A}(\tau)$  be as in (\ref{def:A}). By \eqref{E:p:conditions} and the facts that $\supp \phi_{j} \subset \mathcal{A}_{j}$, $\supp \nabla \phi\subset\mathcal{A}_0$, and $\xi\in\mathcal{A}_j$, we have
 \begin{align}\label{E:meanvalue2}
        |m_{s+\rho,j}(\xi,\eta)|&=\bigg|\int_{0}^{1}\frac{d}{{d\tau}}\left(\phi_{j}(\mathbf{A}(\tau))\abs{\mathbf{A}(\tau)}^{s+\rho}P(\mathbf{A}(\tau))\mathbf{A}(\tau)_{\ell}\right)d\tau\bigg|\notag\\
        &=\bigg|\int_{0}^{1}\bigg\{\left(\nabla \phi(2^{-j}\mathbf{A}(\tau))\cdot(2^{-j}\eta)+(s+\rho)\phi_{j}(\mathbf{A}(\tau))\abs{\mathbf{A}(\tau)}^{-2}\mathbf{A}(\tau)\cdot \eta\right)P(\mathbf{A}(\tau))\mathbf{A}(\tau)_{\ell}\notag\\
        & \qquad\qquad+\phi_{j}(\mathbf{A}(\tau))\left[\nabla P(\mathbf{A}(\tau))\cdot \eta\mathbf{A}(\tau)_{\ell}+P(\mathbf{A}(\tau))\eta_{\ell} \right]\bigg\} \abs{\mathbf{A}(\tau)}^{s+\rho}d\tau\bigg|\notag\\
    &\le C\abs{\eta}\abs{\xi}^{\rho}\int_{0}^{1}\abs{\mathbf{A}(\tau)}^{s}d\tau+ C\abs{\eta}\phi_j(\mathbf{A}(\tau))\int_{0}^{1}\abs{\mathbf{A}(\tau)}^{s+\rho}d\tau\leq C\abs{\eta}\abs{\xi}^\rho\int_0^1|\mathbf{A}(\tau)|^s d\tau.
    \end{align}
By the triangle inequality, we have
\begin{align}\label{est:m:secondcase}
  |m_{s+\rho,j}(\xi,\eta)|\le C(\abs{\xe}^{s}+\abs{\eta}^{s})\abs{\eta}\abs{\xi}^{\rho}.  
\end{align}
Upon returning to \eqref{def:L:comm:relation}, and applying \eqref{est:m:secondcase}, we obtain
\begin{align*}
        |\mathcal{L}^{s+\rho}_{j}(f,g,h)|
        \leq& C\iint\limits_{\xi\in{\Acal}_j}\left(|\xi-\eta|^{s}+|\eta|^{s}\right)|\hat{ f}(\xi-\eta)|\widehat{\Lam g}(\eta)|\abs{\xi}^{\rho}|\hat{h}(\xi)|d\eta d\xi\\
        =&C\iint\limits_{\xi\in{\Acal}_j}|\hat{ f}(\xi-\eta)|\widehat{\Lam^{s+1} g}(\eta)|\abs{\xi}^{\rho}|\hat{h}(\xi)|d\eta d\xi+C\iint\limits_{\xi\in{\Acal}_j}|\Lam^{s}\hat{ f}(\xi-\eta)|\widehat{\Lam g}(\eta)|\abs{\xi}^{\rho}|\hat{h}(\xi)|d\eta d\xi\\
        =&I+II.
        \end{align*}
Applying the Cauchy-Schwarz inequality, Young's convolution inequality, and Bernstein inequality, we obtain
\[|I|\le Cc_{j}\Sob{\hat{f}}{L^1}\Sob{g}{\Hdot^{s+1}}\Sob{h}{\Hdot^{\rho}},\quad |II|\le Cc_{j}\Sob{{f}}{\Hdot^{s}}\Sob{\widehat{\Lam g}}{L^1}\Sob{h}{\Hdot^{\rho}}.\]
Applying \cref{lem:commutator1} with $\si=\eps=s$, we obtain
\[|I|\le Cc_{j}\Sob{f}{\Hdot^s}\Sob{g}{\Hdot^{2}}\Sob{h}{\Hdot^{\rho}},\quad |II|\le Cc_{j}\Sob{{f}}{\Hdot^{1}}\Sob{ g}{\Hdot^{s+1}}\Sob{h}{\Hdot^{\rho}}.\]
Applying \cref{lem:commutator1} with $\si=s, \eps=\nu$, we obtain
\[|I|,|II|\le Cc_{j}\min\left\{ \Sob{{f}}{\Hdot^{1-\nu}}\Sob{g}{\Hdot^{s+1}},\Sob{{f}}{\Hdot^s}\Sob{g}{\Hdot^{2-\nu}}\right\}\Sob{h}{\Hdot^{\rho+\nu}}.\]
Collecting the estimates above, we obtain the required result.
\end{proof}
\section{Estimates for an inhomogeneous linear conservation law with modified flux}\label{sect:mod:flux}
The proof of our main result, \cref{T:theorem1}, will rely on estimates for a linear scalar conservation law whose flux accommodates the commutator structure associated with the skew-adjoint operator $\nabla^{\perp}\Lam^{\beta-2}p(\Lam)$. This commutator structure is crucial to establishing continuity of the corresponding flow map. Given $q$ and $G$ sufficiently smooth, the conservation law and its corresponding initial value problem is given as follows
\begin{align}\label{E:mod:claw}
   \begin{cases}
   \partial_{t}\theta+ \Div F_q(\tht)=G,\\
   \tht(0,x)=\tht_0(x).
   \end{cases}
\end{align}
  where the flux, $F_q(\tht)$, is defined by
\begin{align}\label{def:mod:flux}
F_q(\tht):=(\nabla^{\perp}\Lam^{\be-2}p(\Lam)q) \tht+\Lam^{\be-2}p(\Lam)(({\nabla}^\perp\tht)q)
\end{align}
Observe that, formally, we have $\Div F_{-\tht}(\tht)=-({\nabla}^{\perp}\Lam^{\be-2}p(\Lam)\tht)\cdotp{\nabla}\tht=u\cdotp{\nabla}\tht$. Hence, we recover equation \eqref{E:log-beta} in the particular case when $q=-\tht$.
This modification to the original flux is precisely what allows us to obtain estimates in the space $H^{\be+1}$ that are not otherwise available for the linear transport equation with a Lipschitz regular advecting velocity $u$. Ultimately, \eqref{E:mod:claw} will be exploited to establish continuity of the data-to-solution mapping $\Phi:H^{\be+1}\goesto C([0,T];H^{\be+1})$, $\tht_0\mapsto (\Phi(\tht_0))(t)=\tht(t;\tht_0)$, where $\tht(t;\tht_0)$ represents the solution of \eqref{E:log-beta} corresponding to the initial value problem \eqref{E:log-beta}.

To this end, we must first develop apriori estimates for the system \eqref{E:mod:claw}. These estimates will guarantee its own well-posedness (cf. \cref{T:modclaw:wellpsdn}). The proof of continuity of $\Phi$ will then rely on a continuity-with-respect-to-parameters-type of result (cf. \cref{lem:stability}).

It will be convenient to introduce the following notation:
    \begin{align*}
        A=\Lam^{\be-2}p(\Lam),\quad \Lam^{\si}_{j}=\Lam^{\si}\De_{j}, \quad \bdy^\perp=(-\bdy_2,\bdy_1).
    \end{align*}
We will also make use of the convention that we sum over repeated indices, unless they correspond to Littlewood-Paley operators. Given $q$, let $v$ denote
    \begin{align}\label{def:vel:phi}
        v:=-{\nabla}^{\perp}A q.
    \end{align}
Then we have ${\nabla}\cdotp v=0$; in particular, this implies that
    \begin{align}\label{eq:v:cancel}
        \lb v\cdotp\nabla h,h\rb=0,
    \end{align}
for any sufficiently smooth function $h$. We will first establish $L^2$ space estimates. Then we will proceed to establishing estimates in Sobolev spaces $\Hdot^{\si}$ for $\si \in [1,\be+1]$. 

\subsection{$L^2$ estimates} {Taking the inner product in $L^2$ of \eqref{E:mod:claw}  with $\tht$, we obtain}
	\begin{align}\label{balance} 
	   {\frac{1}{2}\frac{d}{dt}\Sob{\tht}{L^2}^2+\lb \Div F_q(\tht),{\tht}\rb=\lb G ,\tht \rb.}
	\end{align}
Note that $A\bdy_\ell$ is a skew-adjoint operator. As a consequence of this and \eqref{eq:v:cancel}, we see that
\begin{align}\label{E:skew-adjoint}
    \langle \Div F_{q}(\tht),\theta\rangle=-\lb A{\nabla}\cdotp(({\nabla}^\perp q)\tht),\tht\rb=\lb\nabla^\perp q\cdotp\nabla (A\tht),\tht\rb
    &=-\frac{1}{2}\lb [A,\nabla^\perp q\cdotp\nabla]\tht,\tht\rb=-\frac{1}{2}\lb [A\bdy_\ell,(\nabla^\perp q)^\ell]\tht,\tht\rb.
\end{align}
Upon applying \cref{lem:commutator3} with $s=2-\beta$ and $\eps>0$ sufficiently small, we obtain
\begin{align}
    |\langle \Div F_{q}(\tht),\theta\rangle|
    \le C \Sob{q}{H^{\beta+1}}\Sob{\theta}{\Hdot^{\eps}}\Sob{\theta}{L^2}. \label{E:Young}
\end{align}
Returning now to (\ref{balance}), applying the Cauchy-Schwarz inequality to the term on the right-hand side, then invoking \eqref{E:Young}, we arrive at
\begin{align}\label{est:apriori:L2}
    \frac{d}{dt}\Sob{\tht}{L^2}^2\le C\Sob{q}{H^{\beta+1}}\Sob{\theta}{\Hdot^{\eps}}\Sob{\theta}{L^2}+C\Sob{G}{L^2}\Sob{\tht}{L^2}.
\end{align}
\subsection{Homogeneous Sobolev space estimates} 
{Upon applying the operator $\Lam_j^{\si}$ to \eqref{E:mod:claw}, then taking the $L^2$-inner product of the resulting equation with $\Lam_j^{\sigma}\tht$, we obtain}
	\begin{align}\label{eq:balance:basic}
	   {\frac{1}{2}\frac{d}{dt}\Sob{\Lam^{\si}_{j}\tht}{L^2}^2=-\lb \Lam^{\si}_j{\nabla}\cdotp F_q(\tht),\Lam^{\si}_{j}{\tht}\rb+\lb \Lam^{\si}_{j}G ,\Lam^{\si}_{j}\tht\rb=I+II}
	\end{align}
We will first treat $I$. For this, we will distinguish between the two cases, ${\si \in [1,2)}$ and ${\si \in [2,\be+1]}$.

\subsubsection*{Case: $\si \in [1,2)$} In this case, we make use of the fact that $\nabla^\perp q$ is divergence-free and the skew self-adjointness of $A\bdy_\ell$ in order to decompose $I$ as
\begin{align*}
    \lb \Lam^{\si}_j\Div F_q(\tht),\Lam^{\si}_{j}{\tht}\rb&=\underbrace{\lb \Lam^{\si}_j(\nabla^{\perp}Aq\cdot\nabla \tht),\Lam^{\si}_{j}{\tht}\rb}_{I^a}-\underbrace{\lb \Lam^{\si}_jA(\nabla^{\perp}q\cdot\nabla \tht),\Lam^{\si}_{j}{\tht}\rb}_{I^b}\\
    &=I_{1}+I_{2}+I_{3}+I_{4},
\end{align*}
where
\begin{align*}
    I_{1}&=I^a-I_{2}= \lb \Lam^{\si}_j(\nabla^{\perp}Aq\cdot\nabla \tht),\Lam^{\si}_{j}{\tht}\rb-\lb \nabla^{\perp}Aq\cdot\nabla\Lam^{\si}_j \tht,\Lam^{\si}_{j}{\tht}\rb\\
    &=\lb [\Lam^{\si}_{j},\bdy^{\perp}_{\ell}Aq]\bdy_{\ell}\tht,\Lam^{\si}_{j}\tht \rb,\\
    I_{2}&=\lb \nabla^{\perp}Aq\cdot\nabla\Lam^{\si}_j \tht,\Lam^{\si}_{j}{\tht}\rb=0,\\
    I_{3}&=-I^b+I_{4}=-\lb \Lam^{\si}_jA(\nabla^{\perp}q\cdot\nabla \tht),\Lam^{\si}_{j}{\tht}\rb+\lb \nabla^{\perp}q\cdot\nabla{A}^{\frac{1}{2}}\Lam^{\si}_j \tht,{A}^{\frac{1}{2}}\Lam^{\si}_{j}{\tht}\rb\\
    &=-\lb [\Lam^{\si}_{j}{A}^{\frac{1}{2}},\bdy^{\perp}_{\ell}q]\bdy_{\ell}\tht,\Lam^{\si}_{j}{A}^{\frac{1}{2}}\tht \rb,\\
    I_{4}&=\lb \nabla^{\perp}q\cdot\nabla{A}^{\frac{1}{2}}\Lam^{\si}_j \tht,{A}^{\frac{1}{2}}\Lam^{\si}_{j}{\tht}\rb=0.
\end{align*}
Applying \cref{lem:commutator4} with $s=\si-1$, $\rho=0$ and $P=I$, $\mathscr{D}=\Lam$, we obtain
\begin{align*}
    |I_{1}|&\le Cc_{j}(\Sob{\mathcal{F}( \bdy^{\perp}_{\ell}\Lam A q)}{L^1}+\Sob{\bdy^{\perp}_{\ell}Aq}{\Hdot^2})\Sob{{\bdy_{\ell}\tht}}{\Hdot^{\si-1}}\Sob{\Lam^{\si}_{j}\tht}{L^2}.
    \end{align*}
Using the Cauchy-Schwarz inequality, (\ref{def:p}), and Plancherel's theorem we have
\begin{align}\label{est:L1:estimate}
    \Sob{\mathcal{F}( \bdy^{\perp}_{\ell}\Lam A q)}{L^1}&\le  \Sob{\mathbbm{1}_{\mathcal{B}_{0}}(\eta)|\eta|}{L^2}\Sob{\eta_{\ell}^{\perp}|\eta|^{\be-2}p(|\eta|)\hat{q}(\eta)}{L^2} + \Sob{\mathbbm{1}_{\mathcal{B}_{0}^{c}}(\eta)|\eta|^{-1}p(|\eta|)}{L^2}\Sob{\eta_{\ell}^{\perp}|\eta|^{\be}\hat{q}(\eta)}{L^2}\notag\\
    &\le C\Sob{q}{H^{\be+1}}.
\end{align}
Thus
\begin{align*}
    |I_{1}|&\le Cc_{j}\Sob{q}{H^{\be+1}}\Sob{\tht}{\Hdot^{\si}}\Sob{\Lam^{\si}_{j}\tht}{L^2}.
\end{align*}
Applying \cref{lem:commutator4} with $s=\si-1$, $\rho=\frac{\be-2}{2}$, $\nu=2-\be$ and $P=p^{\frac{1}{2}}$, $\mathscr{D}=\Lam$, we obtain
\begin{align*}
    |I_{3}|&\le Cc_{j}\Sob{\bdy^{\perp}_{\ell}q}{\Hdot^{\be}}\Sob{{\bdy_{\ell}\tht}}{\Hdot^{\si-1}}\Sob{{A}^{\frac{1}{2}}\Lam^{\si}_j \tht}{\Hdot^{\frac{2-\be}{2}}}\\
    &\le Cc_{j}\Sob{q}{\Hdot^{\be+1}}\Sob{\tht}{\Hdot^{\si}}\Sob{\Lam^{\si}_{j}\tht}{L^2}.
\end{align*}
\subsubsection*{Case: $\si \in [2,\be+1]$} In this case, we decompose $I$ as
\begin{align*}
    \lb \Lam^{\si}_j \Div F_q(\tht),\Lam^{\si}_{j}{\tht}\rb&=\underbrace{\lb \Lam^{\si}_j(\nabla^{\perp}Aq\cdot\nabla \tht),\Lam^{\si}_{j}{\tht}\rb}_{J^a}-\underbrace{\lb \Lam^{\si}_jA(\nabla^{\perp}q\cdot\nabla \tht),\Lam^{\si}_{j}{\tht}\rb}_{J^b}\\
    &=J_{1}+J_{2}+J_{3}+J_{4}+J_{5},
\end{align*}
where
\begin{align}
	J_{1}=&\underbrace{\lb (\nabla^{\perp}A\Lam^{\si}_{j}{q} \cdot \nabla) \tht,\Lam^{\si}_{j} \tht \rb}_{J_{1}^a} -\underbrace{\lb \nabla^{\perp}A\cdot(\Lam^{\si}_{j}{q}\nabla \tht),\Lam^{\si}_{j} \tht \rb}_{J_{1}^b}\notag\\
	=&\lb [\bdy_{\ell}^{\perp}A,\bdy_{\ell} \tht]\Lam^{\si}_{j}q,\Lam^{\si}_{j}\tht\rb\notag\\
	J_{2}=&\lb \nabla^{\perp}A{q}\cdot {\nabla} \Lam^{\si}_{j} \tht,\Lam^{\si}_{j} \tht \rb=0\notag\\
	J_{3}=&J^a-J_{1}^a-J_{2}\notag\\
	=&\left \{ \lb \Lam^{\si}_{j}(\nabla^{\perp}A{q} \cdot \nabla \theta),\Lam^{\si}_{j} \tht  \rb
	- \lb (\nabla^{\perp}A\Lam^{\si}_{j}{{q}} \cdot \nabla) \tht,\Lam^{\si}_{j} \tht \rb  -\lb \nabla^{\perp}A{q}\cdot {\nabla} \Lam^{\si}_{j} \tht,\Lam^{\si}_{j} \tht \rb \right\}\notag\\
	J_{4}=&\lb (\nabla^{\perp}{q}\cdot {\nabla} A^{\frac{1}{2}}\Lam^{\si}_{j} \tht),A^{\frac{1}{2}}\Lam^{\si}_{j} \tht \rb=0\notag\\
	J_{5}=&-J^b+J_{1}^b+J_{4}\notag\\
	=&- \left \{ \lb \Lam^{\si}_{j} A(\nabla^{\perp}q\cdot\nabla \tht),\Lam^{\si}_{j}{\tht}\rb
	- \lb \nabla^{\perp}A\cdot(\Lam^{\si}_{j}{q}\nabla \tht),\Lam^{\si}_{j} \tht \rb  -\lb (\nabla^{\perp}{q}\cdot {\nabla} A^{\frac{1}{2}}\Lam^{\si}_{j} \tht),A^{\frac{1}{2}}\Lam^{\si}_{j} \tht \rb \right\}\notag
\end{align}
Applying \cref{lem:commutator2} with $s=2-\be$ and $\eps=\be+1-\si$, we obtain
\begin{align*}
    |J_1|&\le C\Sob{\bdy_{\ell}\tht}{\Hdot^{\si-1}}\Sob{\Lam^{\si}_{j}q}{\Hdot^{\be+1-\si}}\Sob{\Lam^{\si}_{j}\tht}{L^2}\\
    &\le Cc_{j}\Sob{q}{\Hdot^{\be+1}}\Sob{\tht}{\Hdot^{\si}}\Sob{\Lam^{\si}_{j}\tht}{L^2},
\end{align*}
where
\[c_{j}=\frac{\Sob{\Lam^{\si}_{j}q}{\Hdot^{\be+1-\si}}}{\Sob{q}{\Hdot^{\be+1}}}\in \ell^{2}(\mathbb{Z}).\]
Now, as in \cite{HuKukavicaZiane2015}, we observe that we may write $J_{3}$ as a double commutator. Indeed, for any $\widetilde{\si}>2$, we have
	\begin{align}\label{split:Lam:rewrite}
	\Lam^{\widetilde{\si}}f=\Lam^{\widetilde{\si}-2}(-\De)f=-(\Lam^{\widetilde{\si}-2}\bdy_{l})\bdy_{l}f.
	\end{align}
	Then by applying \eqref{split:Lam:rewrite}, the product rule, and \eqref{eq:v:cancel}, we have
	\begin{align*}
	    J_{3}=&-\lb\Lam^{\si-2}_{j}\bdy_{l}(\nabla^{\perp}A\bdy_{l}{q} \cdot \nabla \theta),\Lam^{\si}_{j} \tht ,\rb+\lb (\nabla^{\perp}A\Lam^{\si-2}_{j}\bdy_{l}\bdy_{l}{{q}} \cdot \nabla) \tht,\Lam^{\si}_{j} \tht \rb\\
	    &-\lb\Lam^{\si-2}_{j}\bdy_{l}(\nabla^{\perp}A{q} \cdot \nabla \bdy_{l}\theta),\Lam^{\si}_{j} \tht ,\rb+\lb (\nabla^{\perp}A{q}\cdot {\nabla} \Lam^{\si-2}_{j}\bdy_{l}\bdy_{l} \tht),\Lam^{\si}_{j} \tht \rb\\
	    =&-\underbrace{\lb[\Lam_{j}^{\si-2}\bdy_{l},\bdy_{\ell}\tht]\bdy_{\ell}^{\perp}\bdy_{l}A q,\Lam^{\si}_{j} \tht\rb}_{J^a_3}-\underbrace{\lb [\Lam_{j}^{\si-2}\bdy_{l},\bdy_{\ell}^{\perp}A q]\bdy_{\ell}\bdy_{l}\tht ,\Lam_{j}^{\si}\tht\rb}_{J^b_3}
	\end{align*}
Similarly, we can express $J_5$ as
\begin{align*}
    J_{5}=\underbrace{\lb[\Lam_{j}^{\si-2}\bdy_{l},\bdy_{\ell}\tht]\bdy_{\ell}^{\perp}\bdy_{l} q,A\Lam^{\si}_{j} \tht\rb}_{J^a_5}+\underbrace{\lb [\Lam_{j}^{\si-2}A^{\frac{1}{2}}\bdy_{l},\bdy_{\ell}^{\perp} q]\bdy_{\ell}\bdy_{l}\tht ,A^{\frac{1}{2}}\Lam_{j}^{\si}\tht\rb}_{J^b_5}
\end{align*}
Applying \cref{lem:commutator4} with $s=\si-2$, $\rho=0$ and $P=I$, $\mathscr{D}=\bdy_{l}$, then proceeding just like in (\ref{est:L1:estimate}), we obtain
\begin{align*}
    |J^a_3|&\le Cc_{j}(\Sob{\mathcal{F}(\bdy_{\ell}^{\perp}\bdy_{l}A q)}{L^1}+\Sob{\bdy_{\ell}^{\perp}\bdy_{l}A q}{\Hdot^1})\Sob{\bdy_{\ell}\tht}{\Hdot^{\si-1}}\Sob{\Lam^{\si}_{j}\tht}{L^2}\\
    &\le Cc_{j}\Sob{q}{H^{\be+1}}\Sob{\tht}{\Hdot^{\si}}\Sob{\Lam^{\si}_{j}\tht}{L^2}.
\end{align*}
Applying \cref{lem:commutator4} with $s=\si-2$, $\rho=0$ and $P=I$, $\mathscr{D}=\bdy_{l}$, we obtain
\begin{align*}
    |J^b_3|&\le Cc_{j}(\Sob{\mathcal{F}(\Lam\bdy_{\ell}^{\perp}A q)}{L^1}+\Sob{\bdy_{\ell}^{\perp}A q}{\Hdot^2})\Sob{\bdy_{\ell}\bdy_{l}\tht}{\Hdot^{\si-2}}\Sob{\Lam^{\si}_{j}\tht}{L^2}\\
    &\le Cc_{j}\Sob{q}{H^{\be+1}}\Sob{\tht}{\Hdot^{\si}}\Sob{\Lam^{\si}_{j}\tht}{L^2}.
\end{align*}
Applying \cref{lem:commutator4} with $s=\si-2$, $\rho=0$, $\nu=2-\be$ and $P=I$, $\mathscr{D}=\bdy_{l}$, we obtain
\begin{align*}
    |J^a_5|&\le Cc_{j}\Sob{\bdy_{\ell}^{\perp}\bdy_{l}q}{\Hdot^{\be-1}}\Sob{\bdy_{\ell}\tht}{\Hdot^{\si-1}}\Sob{A\Lam^{\si}_{j}\tht}{\Hdot^{2-\be}}\\
    &\le Cc_{j}\Sob{q}{\Hdot^{\be+1}}\Sob{\tht}{\Hdot^{\si}}\Sob{\Lam^{\si}_{j}\tht}{L^2}.
\end{align*}
Applying \cref{lem:commutator4} with $s=\si-2$, $\rho=\frac{\be-2}{2}$, $\nu=2-\be$ and $P=p^{\frac{1}{2}}$, $\mathscr{D}=\bdy_{l}$, we obtain
\begin{align*}
    |J^b_5|&\le Cc_{j}\Sob{\bdy_{\ell}^{\perp}q}{\Hdot^{\be}}\Sob{\bdy_{\ell}\bdy_{l}\tht}{\Hdot^{\si-2}}\Sob{A^{\frac{1}{2}}\Lam^{\si}_{j}\tht}{\Hdot^{\frac{2-\be}{2}}}\\
    &\le Cc_{j}\Sob{q}{\Hdot^{\be+1}}\Sob{\tht}{\Hdot^{\si}}\Sob{\Lam^{\si}_{j}\tht}{L^2}.
\end{align*}

\subsubsection*{Summary of estimates}
Upon returning to \eqref{eq:balance:basic}, collecting the above estimates for $I_1, I_3$ and $J_1, J_3, J_5$, then applying the Cauchy-Schwarz inequality, we obtain
    \[  
        \frac{d}{dt}\Sob{\Lam^{\si}_{j}\tht}{L^2}^2\le Cc_{j}\Sob{q}{H^{\be+1}}\Sob{\tht}{\Hdot^{\si}}\Sob{\Lam^{\si}_{j}\tht}{L^2}+C\Sob{\Lam^{\si}_{j}G}{L^2}\Sob{\Lam^{\si}_{j}\tht}{L^2}.
    \]
Finally summing over $j$ and using the Cauchy-Schwarz inequality, we have
\begin{align}\label{est:apriori:sigmadot}
     \frac{d}{dt}\Sob{\tht}{\Hdot^{\si}}^2\le C\Sob{q}{H^{\be+1}}\Sob{\tht}{\Hdot^{\si}}^{2}+C\Sob{G}{\Hdot^{\si}}\Sob{\tht}{\Hdot^{\si}}.
\end{align}
Combining with the estimates \eqref{est:apriori:L2} and \eqref{est:apriori:sigmadot}, we deduce
\begin{align}\label{est:apriori:sigma}
    \frac{d}{dt}\Sob{\tht}{H^{\si}}^2\le C\Sob{q}{H^{\be+1}}\Sob{\tht}{H^{\si}}^{2}+C\Sob{G}{H^{\si}}\Sob{\tht}{H^{\si}}.
\end{align}
An application of Gronwall's inequality then yields
\begin{align}\label{est:Gronwall}
    \Sob{\tht(t)}{H^{\si}}\le e^{C\Sob{q}{L^{1}_{t}H^{\be+1}}}(\Sob{\tht_{0}}{H^{\si}}+\Sob{G}{L^{1}_{t}H^{\si}}).
\end{align}

From the estimates developed above, one may carry out an artificial viscosity argument to establish the corresponding well-posedness result for \eqref{E:mod:claw}. We state this below as \cref{T:modclaw:wellpsdn}. Details of the artificial viscosity argument are provided in \cref{sect:app}.

\begin{Thm}{\label{T:modclaw:wellpsdn}}
Let $\be\in(1,2)$. Given $T>0$, suppose $q\in L^{1}(0,T;{H}^{\be+1})\cap L^{p}(0,T; H^{-m})$ for some $m>0$ and $p>1$, and $G \in L^{1}(0,T;H^{\si})$ for some  $\si \in [1,\be+1]$. For each $\tht_0\in H^{\si}$, there exists a unique solution $\theta\in C([0,T];H^{\si})$ of \eqref{E:mod:claw} satisfying \eqref{est:Gronwall}.
\end{Thm}

We now establish a ``stability in parameters" type result for \eqref{E:mod:claw}. To prove this, we will require an estimate in $H^{\be}$ for the divergence-term in \eqref{E:mod:claw}; this estimate will also be invoked to establish the well-posedness of \eqref{E:log-beta}.

\begin{Lem}\label{lem:nonlinearest}
        Let $\beta \in (1,2)$. Let $q \in H^{\be}$ and $\theta \in H^{\be+1}$. Let $F_q(\tht)$ be defined as in \eqref{def:mod:flux}. Then
        \begin{align*}
            \Sob{\Div F_{q}(\tht)}{H^{\be}}\le C\Sob{q}{H^{\be}}\Sob{\tht}{H^{\be+1}}.
        \end{align*}
\end{Lem}
\begin{proof}
Let ${H}:=\Div F_{q}(\tht)$. Observe that by \cref{lem:commutator3} with $s=2-\beta$ and $\eps>0$, we have
\begin{align}\label{est:H:L2}
\Sob{H}{L^2}^{2}&=\lb \nabla \cdot((\nabla^{\perp}Aq)\tht),H\rb+\lb A\nabla \cdot((\nabla^{\perp}\tht)q),H\rb\notag\\
&=-\lb[\bdy^{\perp}_{\ell}A,\bdy_{\ell}\tht]q ,H\rb\notag\\
&\le C\Sob{\tht}{H^{\be+1}}(\Sob{q}{\Hdot^\eps}\Sob{H}{L^2}+\Sob{q}{L^2}\Sob{H}{\Hdot^{\eps}}).
\end{align}

On the other hand, one can verify that
\begin{align*}
    \Sob{\Lam^\be_j H}{L^2}^2
    &=\lb \Lam^{\be}_j(\nabla^{\perp}A q\cdot\nabla \tht),\Lam^{\be}_{j}{H}\rb-\lb \Lam^{\be}_{j}A(\nabla^{\perp}q\cdot\nabla \tht),\Lam^{\be}_{j}{H}\rb\\
    &= -\underbrace{\lb[\bdy^{\perp}_{\ell}A,\bdy_{\ell}\tht]\Lam^{\beta}_{j}q ,\Lam^{\beta}_{j}H\rb}_{I}+\underbrace{\lb[\Lam^{\beta}_{j}, \bdy_{\ell}\tht]\bdy^{\perp}_{\ell}A q ,\Lam^{\beta}_{j}H\rb}_{II}-\underbrace{\lb[\Lam^{\be}_{j},\bdy_{\ell}\tht]\bdy^{\perp}_{\ell}q ,A\Lam^{\beta}_{j} H\rb}_{III}.
\end{align*}
Applying \cref{lem:commutator2} with $s=2-\beta$ and $\eps=0$, we obtain
\begin{align*}
    |I|&\le C\Sob{\bdy_{\ell}\theta}{\Hdot^{\be}}\Sob{\Lam^{\be}_{j}q}{L^2}\Sob{\Lam^{\be}_{j}H}{L^2}\\
    &\le Cc_{j}\Sob{q}{\Hdot^{\be}}\Sob{\tht}{\Hdot^{\be+1}}\Sob{\Lam^{\be}_{j}H}{L^2},
\end{align*}
where
\[c_{j}=\frac{\Sob{\Lam^{\be}_{j}q}{L^2}}{\Sob{q}{\Hdot^{\be}}}\in \ell^{2}(\mathbb{Z}).\]
Applying \cref{lem:commutator4} with $s=\beta-1$, $\rho=0$ and $P=I$, $\mathscr{D}=\Lam$, and proceeding just like in (\ref{est:L1:estimate}), we obtain
\begin{align*}
    |II|&\le Cc_{j}\Sob{\bdy_{\ell}\tht}{\Hdot^{\be}}(\Sob{\mathcal{F}(\bdy^{\perp}_{\ell}A q)}{L^1}+\Sob{\bdy^{\perp}_{\ell}A q}{\Hdot^1})\Sob{\Lam^{\be}_{j}H}{L^2}\\
    &\le Cc_{j}\Sob{q}{H^{\be}}\Sob{\tht}{\Hdot^{\be+1}}\Sob{\Lam^{\be}_{j}H}{L^2},
\end{align*}
Applying \cref{lem:commutator4} with $s=\beta-1$, $\rho=0$, $\nu=2-\beta$ and $P=I$, $\mathscr{D}=\Lam$, we obtain
\begin{align*}
    |III|&\le Cc_{j}\Sob{\bdy_{\ell}\tht}{\Hdot^{\be}}\Sob{\bdy^{\perp}_{\ell}q}{\Hdot^{\be-1}}\Sob{A\Lam^{\be}_{j}H}{\Hdot^{2-\beta}}\\
    &\le Cc_{j}\Sob{q}{H^{\be}}\Sob{\tht}{\Hdot^{\be+1}}\Sob{\Lam^{\be}_{j}H}{L^2},
\end{align*}
From estimates of $I$--$III$, we deduce
\begin{align}\label{est:H:beta:dot}\Sob{H}{\Hdot^{\be}}\le C\Sob{q}{H^{\be}}\Sob{\tht}{\Hdot^{\be+1}}.
\end{align}

Finally, upon combining \eqref{est:H:beta:dot} and \eqref{est:H:L2}, we conclude
\begin{align}\label{est:H:beta}
    \Sob{H}{H^{\be}}\le C\Sob{q}{H^{\be}}\Sob{\tht}{H^{\be+1}}
\end{align}
as claimed.
\end{proof}

We are now ready to prove the stability-type result for \eqref{E:mod:claw} alluded to earlier, which constitutes the main ingredient for establishing the continuity of the data-to-solution map $\Phi$ of \eqref{E:log-beta}.
\begin{Lem}\label{lem:stability}
Let $\be \in (1,2)$. Let $\{q^{n}\}$ be a sequence of functions satisfying
\begin{align*}
  \sup_{n>0}\Sob{q^{n}}{L^\infty_{T}H^{\be+1}}\le C\quad\text{and}\quad
  \lim_{n \to \infty}\Sob{q^{n}-q^{\infty}}{L^1_{T}H^{\be}}=0,
\end{align*}
for some constant $C>0$, depending on $\be, T$. Given $\tht_{0}\in H^{\be}$ and $G\in L^1(0,T;H^{\be})$, let $\tht^n$ denote the solution of
\begin{align}\label{E:modclaw:n}
\begin{cases}
 \partial_{t}\theta^{n}+ \Div F_{q^{n}}(\tht^{n})=G,\\
 \tht^{n}(0,x)=\tht_0(x).
\end{cases}
\end{align}
for each $n\in \mathbb{N}\cup \{\infty\}$. Then
\[\lim_{n \to \infty}\Sob{\tht^{n}-\tht^{\infty}}{L^{\infty}_{T}H^{\be}}=0.\]
\end{Lem}

\begin{proof}
First, let us consider the case when $\tht_{0} \in H^{\be+1}$ and $G \in L^{1}(0,T;H^{\be+1})$. By \cref{T:modclaw:wellpsdn}, the sequence of solutions to \eqref{E:modclaw:n}, denoted by $\{\tht^{n}\}$, is bounded in $L^{\infty}(0,T;H^{\be+1})$, uniformly in $n$. Observe that $\tht^{n}-\tht^{\infty}$ satisfies
\begin{align}\label{E:modclaw:intermediate}
    \partial_{t}(\tht^{n}-\tht^{\infty})+ \Div F_{q^{n}}(\tht^{n}-\tht^{\infty})=\Div F_{(q^{\infty}-q^{n})}(\tht^{\infty}).
\end{align} 
By \cref{lem:nonlinearest}, we have
\begin{align*}
    \Sob{\Div F_{(q^{\infty}-q^{n})}(\tht^{\infty})}{H^{\beta}}\le C\Sob{q^{\infty}-q^{n}}{H^{\be}}\Sob{\tht^{\infty}}{H^{\be+1}}.
\end{align*}
Upon returning to (\ref{E:modclaw:intermediate}) and using \eqref{est:Gronwall}, we obtain
\[\Sob{\tht^{n}(t)-\tht^{\infty}(t)}{H^{\beta}}\le e^{C\Sob{q^{n}}{L^1_{T}H^{\be+1}}}\Sob{\tht^{\infty}}{L^{\infty}_{T}H^{\be+1}}\Sob{q^{n}-q^{\infty}}{L^1_{T}H^{\be}}.\]
This implies
\[\lim_{n\to \infty}\Sob{\tht^{n}-\tht^{\infty}}{L^{\infty}(0,T;H^{\be})}=0.\]
Now, let us assume that  $\tht_{0}\in H^{\be}$. For all $n \in \mathbb{N}\cup\{\infty\}$ and $k\in \mathbb{N}$, denote by $\tht^{n}_{k}$, the solution to
\begin{align} \label{E:modclaw:k}\partial_{t}\theta^{n}_{k}+ \Div F_{q^{n}}(\tht^{n}_{k})=P_{k}G ,\quad\tht^{n}_{k}(0,x)=P_{k}\tht_0(x),
\end{align}
where $P_{k}$ denotes projection to frequencies $|\xi|\leq 2^{k}$, i.e., $\mathcal{F}({P_{k}f})(\xi)=\mathbbm{1}_{\Bcal_{k}}(\xi)\hat{f}(\xi)$. 
From \eqref{E:modclaw:n} and \eqref{E:modclaw:k}, it follows that
    \[
    \partial_{t}(\tht^{n}-\theta^{n}_{k})+ \Div F_{q^{n}}(\tht^{n}-\tht^{n}_{k})=(I-P_{k})G.
    \]
Then by \cref{T:modclaw:wellpsdn}, we obtain
\begin{align}\label{est:modclaw:intermediate:hf}\Sob{\tht^{n}-\tht^{n}_{k}}{L^{\infty}_{T}H^{\beta}}\le e^{C\Sob{q^{n}}{L^1_{T}H^{\be+1}}}\left(\Sob{(I-P_{k})\tht_{0}}{H^{\beta}}+\Sob{(I-P_{k}G)}{L^1_{T}H^{\be}}\right),\quad \text{for all}\ n \in \mathbb{N}\cup\{\infty\}
\end{align}
Let $\delta>0$. By \eqref{est:modclaw:intermediate:hf}, we can select $k$ large enough, say $k_{0}$, such that $\Sob{\tht^{n}-\tht^{n}_{k_0}}{L^{\infty}_{T}H^{\beta}}\le \delta/3$ for all $n$. Since $P_{k_0}\tht_{0}\in H^{\be+1}$ and $ P_{k_0}G \in L^{1}(0,T;H^{\be+1})$, we can find an integer $N$ such that for all $n\ge N$
    \[
        \Sob{\tht^{n}_{k_0}-\tht^{\infty}_{k_0}}{L^{\infty}_{T}H^{\beta}}\le \delta/3.
    \]
By the triangle inequality, we obtain
    \[
        \Sob{\tht^{n}-\tht^{\infty}}{L^{\infty}_{T}H^{\be}}\le \delta.
    \]
Since $\delta$ was arbitrary, the proof is complete.
\end{proof}
\section{Existence, Uniqueness, and Continuity of solutions to (\ref{E:log-beta})}\label{proof:main}

We will now prove the main theorem of the paper, \cref{T:theorem1}. We will do this in three steps, starting by establishing existence, proving uniqueness, and concluding with continuity with respect to initial data, that is, continuity of the data-to-solution map.

\subsection{Existence}
Observe that since $\nabla^{\perp}\tht\cdot\nabla \tht=0$, we can express equation \eqref{E:beta} as 
\begin{align}\label{E:beta:mod:form}
    \partial_{t}\theta+ \Div F_{-\tht}(\tht)=0 ,\quad\tht(0,x)=\tht_0(x),
\end{align}
where $F$ is as defined in \eqref{def:mod:flux}. By invoking \eqref{est:apriori:sigma} for $\si=\be+1$, we obtain
\begin{align*}
    \frac{d}{dt}\Sob{\tht}{H^{\be+1}}^2\le C\Sob{\tht}{H^{\be+1}}^{3}.
\end{align*}
We conclude that there exists a time $T=T(\Sob{\tht_0}{H^{\be+1}})$ such that $\tht(t,x)$ satisfies
\begin{align}\label{ineq:Sobolev:bound}
    \Sob{\tht}{L^{\infty}_{T}H^{\be+1}}\le 2\Sob{\tht_{0}}{H^{\be+1}}.
\end{align}
The existence of a solution $\tht(t,x)$ can now be established from a standard argument via artificial viscosity similar to the proof of \cref{T:modclaw:wellpsdn}.
\subsection{Uniqueness}Let $\tht_0\in H^{\be+1}$ and suppose $\tht^{(1)},\tht^{(2)}\in C([0,T];H^{\be+1})$ are two solutions of \eqref{E:log-beta} corresponding to $\tht_0$. Let  $\bar{\tht}:=\tht^{(1)}-\tht^{(2)}$ and observe that $\bar{\tht}$ satisfies
\begin{align*}
\begin{cases}
 \partial_{t}\bar{\theta}+ \Div F_{-\tht^{(1)}}(\bar{\tht})=\Div F_{\bar{\tht}}(\tht^{(2)}),\\
 \bar{\tht}(0,x)=0.
\end{cases}
\end{align*}
By \cref{lem:nonlinearest}, we have
\[\Sob{\Div F_{\bar{\tht}}(\tht^{(2)})}{H^{\be}}\le C\Sob{\bar{\tht}}{H^{\be}}\Sob{\tht^{(2)}}{H^{\be+1}}.\]
Using \eqref{est:apriori:sigma}, we obtain
\begin{align}\label{est:stability:Hbeta}
    \frac{d}{dt}\Sob{\bar{\tht}}{H^{\be}}^2\le C\left(\Sob{\tht^{(1)}}{H^{\be+1}}+\Sob{\tht^{(2)}}{H^{\be+1}}\right)\Sob{\bar{\tht}}{H^{\be}}^{2}.
\end{align}
An application of the Gronwall inequality then establishes uniqueness.
\subsection{Continuity of the flow map} 
 Let $\Phi:H^{\be+1}\goesto\bigcup_{T>0}C([0,T];H^{\be+1})$ denote the flow map of \eqref{E:log-beta}. By the previous considerations, we have shown that $\Phi$ is well-defined. We will now show that $\Phi$ is continuous.
\begin{Prop}\label{prop:continuity:flowmap}
\emph{Let $\tht_{0} \in H^{\be+1}$. There exists a neighborhood $U_0\subset H^{\be+1}$ of $\tht_{0}$ and time $T>0$ such that for any sequence of initial data $\{\tht^{n}_{0}\}\subset U_0$}, we have
    \[
        \lim_{n \to \infty} \Sob{\tht^{n}_{0}-\tht_{0}}{H^{\be+1}}=0\quad\text{implies}\quad\lim_{n \to \infty}\Sob{\Phi(\tht_{0}^{n})-\Phi(\tht_{0})}{L^{\infty}_{T}H^{\be+1}}=0.
    \]
\end{Prop} 
\begin{proof}
Let $\tht=\Phi(\tht_0)$ and $T_0>0$ denote the local existence time of $\tht$. Define $U_0$ by
    \[
        U_0:=\{\tht'_{0}: \quad \Sob{\tht'_{0}-\tht_{0}}{H^{\be+1}}<\Sob{\tht_{0}}{H^{\be+1}}\}.
    \]
Let $K:=4\Sob{\tht_{0}}{H^{\be+1}}$. Denote by $\tht^{n}$ the solution to \eqref{E:beta:mod:form} with the initial data $\tht^n_0$. By \eqref{ineq:Sobolev:bound}, we have
    \[
        \sup_{n>0}\Sob{\tht^n}{L^{\infty}_{T}H^{\be+1}}\le K,
    \]
{for some $0<T\leq T_0$, depending on $\Sob{\tht_0}{H^{\be+1}}$}.
{Then by} the stability estimate in \eqref{est:stability:Hbeta}, we have
    \[
        \lim_{n \to \infty}\Sob{\tht^{n}-\tht}{L^\infty_{T}H^{\be}}=0.
    \]
In particular, we may set $\tht^\infty=\tht$. Observe that to complete the proof, it thus suffices to show that $\nabla \tht^n \to \nabla \tht$ in $L^\infty_{T}H^{\be}$. 

We decompose $\bdy_{\ell}\tht^{n}=\omega_{\ell}^{n}+\zeta_{\ell}^{n}$, where  ($\omega^{n}_{\ell},\zeta^{n}_{\ell}$) satisfy the following equations
\begin{align}\label{E:omega}
    \begin{cases}
    \partial_{t}\om^{n}_{\ell}+ \Div F_{-\tht^{n}}(\om^{n}_{\ell})=G_{\ell}\\
    \om^{n}_{\ell}(0,x)=\bdy_{\ell}\tht_0(x),
    \end{cases}
\end{align}
and 
\begin{align}\label{E:zeta}
   \begin{cases}
   \partial_{t}\ze^{n}_{\ell}+ \Div F_{-\tht^n}(\ze^{n}_{\ell})=G_{\ell}^{n}-G_{\ell}\\
   \ze^{n}_{\ell}(0,x)=\bdy_{\ell}\tht^{n}_0(x)-\bdy_{\ell}\tht_{0}(x),
   \end{cases} 
\end{align}
for $n\in\mathbb{N}\cup\{\infty\}$ and $\ell=1,2$, where
    \[
        G_{\ell}=\Div F_{\bdy_{\ell}\tht}(\tht), \quad G^{n}_{\ell}=\Div F_{\bdy_{\ell}\tht^n}(\tht^{n}).
    \]
Let
    \[
       \om^n:=(\om^n_1,\om^n_2),\quad\om:=(\om_{1},\om_2),\quad \ze^n:=(\ze^n_1,\ze^n_2),\quad \ze:=(\ze_1,\ze_2).
    \]
By \cref{lem:stability}, it follows that
\begin{align*}
   \lim_{n \to \infty} \Sob{\om^{n}-\om}{L^{\infty}_{T}H^{\be}}=0.
\end{align*}
We write
\begin{align*}
    G_{\ell}^{n}-G_{\ell}=\Div F_{\bdy_{\ell}\tht^{n}}(\tht^n-\tht)+\Div F_{(\bdy_{\ell}\tht^{n}-\bdy_{\ell}\tht)}(\tht).
\end{align*}
By \cref{lem:nonlinearest}, we obtain
\begin{align*}
    \Sob{G_{\ell}^{n}-G_{\ell}}{H^{\be}}&\le C\left(\Sob{\bdy_{\ell}\tht^{n}}{H^\be}\Sob{\tht^n-\tht}{H^{\be+1}}+\Sob{\bdy_{\ell}\tht^{n}-\bdy_{\ell}\tht}{H^\be}\Sob{\tht}{H^{\be+1}} \right)\\
    &\le CK\left(\Sob{\ze^{n}}{H^\be}+\Sob{\om^{n}-\om}{H^\be}\right).
\end{align*}
Using the above estimate along with \eqref{est:Gronwall}, we obtain
\begin{align*}
    \Sob{\ze^{n}(t)}{H^{\be}}\le e^{CKt}\left\{\Sob{\tht^{n}_{0}-\tht_{0}}{H^{\be+1}}+CK\int_{0}^{t}\left(\Sob{\ze^{n}(s)}{H^\be}+\Sob{\om^{n}(s)-\om(s)}{H^\be}\right)ds\right\}.
\end{align*}
Applying Gronwall's inequality, then passing to the limit as $n \to \infty$, we obtain
\[\limsup_{n \to \infty}\Sob{\ze^{n}(t)}{L^{\infty}_{T}H^{\be}}\le \lim_{n \to \infty} e^{CKT}\left(\Sob{\tht^{n}_{0}-\tht_{0}}{H^{\be+1}}+CKT\Sob{\om^{n}-\om}{L^{\infty}_{T}H^{\be}}\right)=0.\]
Finally, by the triangle inequality
    \[
       \limsup_{n\goesto\infty}\Sob{\nabla \tht^{n}-\nabla \tht}{L^{\infty}_{T}H^\be}\leq \limsup_{n\goesto\infty}\Sob{\om^n-\om}{L^\infty_TH^\be}+\limsup_{n\goesto\infty}\Sob{\ze^n-\ze}{L^\infty_TH^{\be}}=0,
    \]
as claimed.
\end{proof}

\appendix
\renewcommand{\thesection}{\null}
\renewcommand{\theequation}{A.\arabic{equation}}
\section{}\label{sect:app}
\begin{proof}[Proof of \cref{lem:elem:convex}]
First observe that by the Cauchy-Schwarz inequality
\begin{align*}
\abs{\varphi+\tau \vartheta}^{2}=\abs{\varphi}^{2}+\tau^{2}+2\tau\varphi\cdot\vartheta\ge\abs{\varphi}^{2}+\tau^{2}-2\tau\abs{\varphi}=\abs{\abs{\varphi}-\tau}^{2}.
\end{align*}
Since $s>0$, it follows that
\begin{align*}
\int_{0}^{1}\frac{1}{\abs{\varphi+\tau \vartheta}^{s}}\,d\tau \le\int_{0}^{1}\frac{1}{\abs{\abs{\varphi}-\tau}^{s}}d\tau.
\end{align*}
If $\abs{\varphi}\le 1$, then
\begin{align*}
\begin{split}
\int_{0}^{1}\frac{1}{\abs{\abs{\varphi}-\tau}^{s}}\,d\tau&=\int_{0}^{\abs{\varphi}}\frac{1}{(\abs{\varphi}-\tau)^{s}}\,d\tau+\int_{\abs{\varphi}}^{1}\frac{1}{(\tau-\abs{\varphi})^{s}}\,d\tau\\
&\le C(\abs{\varphi}^{1-s}+(1-\abs{\varphi})^{1-s})\le 2C.
\end{split}
\end{align*}
If $\abs{\varphi}>1$, then since $0<s<1$, we have
\begin{align*}
\int_{0}^{1}\frac{1}{\abs{\abs{\varphi}-\tau}^{s}}\,d\tau=C(-(\abs{\varphi}-1)^{1-s}+\abs{\varphi}^{1-s}),
\end{align*}
which is clearly bounded when $1<\abs{\varphi}\leq2$, and also when $\abs{\varphi}>2$ by the mean value theorem.
\end{proof}

We supply the details of the artificial viscosity argument used for proving \cref{T:modclaw:wellpsdn}.

\begin{proof}[Proof of \cref{T:modclaw:wellpsdn}]
First, we mollify $q$ and $G$ with respect to time by setting
\[q^{n}=\rho_{n}\star q,\quad G^{n}=\rho_{n}\star G,\]
where $\{\rho_{n}(t)\}$ is a sequence of standard mollifiers. Observe that we have
\[q^{n}\in C([0,T];H^{\be+1}),\quad G^{n}\in C([0,T];H^{\si}).\] Moreover, $\{q^{n}\}$ is uniformly bounded in $L^{1}([0,T];H^{\be+1})$ and $\{G^{n}\}$ is uniformly bounded in $L^{1}([0,T];H^{\si})$. Now, let us consider the following artificial viscosity regularization of \eqref{E:mod:claw}: 

\begin{align}\label{E:forced-transport:reg}
\begin{cases}
\partial_{t}\theta^{n}-\frac{1}{n}\De \tht^{n}+ \Div F_{q^{n}}(\tht^{n})=G^{n}.   \\
\tht^{n}(0,x)=\tht_{0}(x).
\end{cases}
\end{align}

For $0\le t\le T$, define
\begin{align*}
    &{F}_{1}(G^{n}):=\int^{t}_{0}e^{\frac{1}{n}\De(t-s)}G^{n}(s)\, ds,\\
    &{F}_{2}(\tht^{n};q^{n}):=\int^{t}_{0}e^{\frac{1}{n}\De(t-s)}\Div F_{q^n}(\tht^n)\,ds.
\end{align*}

We have
\begin{align*}
   \Sob{F_1(G^n)(t)}{H^{\si}}&\le T\Sob{G^n}{L^\infty_{T}H^{\si}}.
\end{align*}

To estimate $\Sob{{F}_{2}(\tht^n;q^n)}{H^{\si}}$, we consider the two cases $\si \in [1,2)$ and $\si \in [2,\be+1]$ separately.
\subsubsection*{\textbf{Case: ${\si}\in [1,2)$}} Using Plancherel's theorem and applying \cref{T:Sobolev}, we have
\begin{align*}
    &\Sob{F_{2}(\tht^n;q^n)(t)}{H^{\si}}\\&\le \int_{0}^{t}\left(\Sob{(I-\De)^{\si/2}e^{\frac{1}{n}\De(t-s)}(\Lam^{\be-2}p(\Lam)\bdy^{\perp}_{\ell}{q^n}\bdy_{\ell}\tht^n)}{L^2}+\Sob{(I-\De)^{\si/2}e^{\frac{1}{n}\De(t-s)}\Lam^{\be-2}p(\Lam)(\bdy^{\perp}_{\ell}{q^n}\bdy_{\ell}\tht^n)}{L^2}\right)ds \\
    &\le \int_{0}^{t}C\left\{1+\frac{n^{\frac{\si}{2}}}{(t-s)^{\frac{\si}{2}}}\right\}\left(\Sob{\Lam^{\be-2}p(\Lam)\bdy^{\perp}_{\ell}{q^n}\bdy_{\ell}\tht^n}{L^2}+\Sob{\bdy^{\perp}_{\ell}{q^n}\bdy_{\ell}\tht^n}{\Hdot^{\be-2}}\right)ds\\& \le C\left\{T+n^{\frac{\si}{2}}T^{\frac{2-\si}{2}}\right\}\left(\Sob{\Lam^{\be-2}p(\Lam)\bdy^{\perp}_{\ell}{q^n}}{L^{\infty}_{T}L^{\infty}}\Sob{\bdy_{\ell}\tht^n}{L^{\infty}_{T}L^2}+\Sob{\bdy^{\perp}_{\ell}{q^n}}{L^{\infty}_{T}\Hdot^{\be-\si}}\Sob{\bdy_{\ell}\tht^n}{L^{\infty}_{T}\Hdot^{\si-1}}\right)\\
    &\le C\left\{T+n^{\frac{\si}{2}}T^{\frac{2-\si}{2}}\right\}\Sob{q^n}{L^{\infty}_{T}H^{\be}}\Sob{\tht^n}{L^{\infty}_{T}H^{\si}}.
\end{align*}
\subsubsection*{\textbf{Case: ${\si}\in [2,\be+1]$}}  We proceed just like above and apply \cref{T:Sobolev} to obtain
\begin{align*}
    &\Sob{F_{2}(\tht^n;q^n)(t)}{\Hdot^{\si}}\\&\le Cn^{\frac{\si+1}{4}}\int_{0}^{t}\left\{\frac{1}{(t-s)^{\frac{\si+1}{4}}}\right\}\left(\Sob{\Lam^{\be-2}p(\Lam)\bdy^{\perp}_{\ell}{q^n}\bdy_{\ell}\tht^n}{\Hdot^{\frac{\si-1}{2}}}+\Sob{\bdy^{\perp}_{\ell}{q^n}\bdy_{\ell}\tht^n}{\Hdot^{\be-2+\frac{\si-1}{2}}}\right)ds\\ &\le Cn^{\frac{\si+1}{4}}T^{\frac{3-\si}{4}}\left(\Sob{\Lam^{\be-2}p(\Lam)\bdy^{\perp}_{\ell}{q^n}}{L^{\infty}_{T}\Hdot^{\frac{\si+1}{4}}}\Sob{\bdy_{\ell}\tht^n}{L^{\infty}_{T}\Hdot^{\frac{\si+1}{4}}}+\Sob{\bdy^{\perp}_{\ell}{q^n}}{L^{\infty}_{T}\Hdot^{\be-2+\frac{\si+1}{4}}}\Sob{\bdy_{\ell}\tht^n}{L^{\infty}_{T}\Hdot^{\frac{\si+1}{4}}}\right)\\
    &\le Cn^{\frac{\si+1}{4}}T^{\frac{3-\si}{4}}\Sob{q^n}{L^{\infty}_{T}H^{\be}}\Sob{\tht^n}{L^{\infty}_{T}H^{\si}}.
\end{align*}
Similarly,  we have
\begin{align*}
    \Sob{F_{2}(\tht^n;q^n)(t)}{L^2} \le CT\Sob{q^n}{L^{\infty}_{T}H^{\be}}\Sob{\tht^n}{L^{\infty}_{T}H^{\si}}.
\end{align*}

Using Picard's theorem \cite{Lemarie-Rieusset2002}, there exists a unique solution $\tht^{n}$ to (\ref{E:forced-transport:reg}) such that $\tht^{n}\in L^{\infty}_{T^{n}}H^{\si}$ for some time $T^{n}>0$. Owing to the uniform estimate in \eqref{est:Gronwall}, we can conclude that
\begin{align*}
    T^{n}=T,\quad \text{for all}\ n.
\end{align*}
Let us denote by
\[\bar{\tht}^{n}(t)=\tht^{n}(t)-\int_{0}^{t}G^{n}(s)ds.\]
Then, by \eqref{est:Gronwall}, $\Sob{\bar{\tht}^{n}}{L^{\infty}_{T}H^{\si}}$ is bounded uniformly in $n$.
Using similar methods as above, it is easy to see that $\Sob{\partial_{t}\bar{\theta}^{n}}{L^{p}_{T}H^{-m}}$ is bounded uniformly in $n$, for some sufficiently large $m>0$. By an application of the Aubin-Lions lemma \cite{ConstantinFoiasBook1988}, there exists $\bar{\tht}\in L^{\infty}([0,T];H^{\si})$ such that for any given $\varphi \in C_{c}^{\infty}([0,T] \times \mathbb{R}^2)$, there exists a subsequence of $\{\bar{\tht}^n\}$, denoted by $\{\bar{\tht}^{n_{k}}\}$ satisfying
\begin{align*}
   & \bar{\tht}^{n_{k}}\xrightharpoonup {\text{w*}}\bar{\tht}\quad\text{in}\quad L^{\infty}([0,T];H^{\si}),\\
    & \varphi\bar{\tht}^{n_{k}} \longrightarrow \varphi\bar{\tht}\quad \text{in} \quad C([0,T];H^{\si-\eps}),
\end{align*}
for any $\eps>0$. It is then straightforward to show that $\tht(t)=\bar{\tht}(t)+\int_{0}^{t}G(s)ds$ is a weak solution of \eqref{E:mod:claw}. This completes the proof. 
\end{proof}

\begin{footnotesize}
\bibliographystyle{plain}
\bibliography{Main.bib}
\end{footnotesize}

\vspace{.3in}
\begin{multicols}{2}

\noindent Michael S. Jolly\\ 
{\footnotesize
Department of Mathematics\\
Indiana University-Bloomington\\
Web: \url{https://msjolly.pages.iu.edu/}\\
 Email: \url{msjolly@indiana.edu}}\\[.2cm]

\noindent Anuj Kumar\\ 
{\footnotesize
Department of Mathematics\\
Indiana University-Bloomington\\
Web: \url{https://math.indiana.edu/about/graduate-students}\\
 Email: \url{kumar22@iu.edu}} \\[.2cm]

\columnbreak 

\noindent Vincent R. Martinez\\
{\footnotesize
Department of Mathematics and Statistics\\
CUNY-Hunter College\\
Web: \url{http://math.hunter.cuny.edu/vmartine/}\\
 Email: \url{vrmartinez@hunter.cuny.edu}}

\end{multicols}

\end{document}